\documentclass[11pt]{amsart} 
\usepackage[curve,v2]{xy}
\usepackage{amssymb,latexsym,amsfonts, amsthm, amsmath}

\newcommand{\tensor}{\otimes}

\newcommand{\Tor}{\operatorname{Tor}}

\renewcommand{\H}{\operatorname{Hilb}^2\hspace{-.05in}X}

\newcommand{\wtp}{\widetilde{\P}^n}

\newcommand{\wts}{\widetilde{\Sigma}}

\newcommand{\ses}[3]{0\rightarrow#1\rightarrow#2
   \rightarrow#3\rightarrow0}

\newcommand{\PP}{\ensuremath{\mathbb{P}}}

\newcommand{\E}{{\mathcal E}}

\newcommand{\I}{{\mathcal I}}

\renewcommand{\P}{{\mathbb{P}}}

\newcommand{\NS}{{N_p^{\Sigma}}}

\renewenvironment{proof}{\par \medskip \noindent
{\sc Proof:}}{}

\newtheorem{thm}{Theorem}[section]   
\newtheorem{cor}[thm]{Corollary}     
\newtheorem{lemma}[thm]{Lemma}         
\newtheorem{prop}[thm]{Proposition}  

\theoremstyle{definition} 
\newtheorem{conj}[thm]{Conjecture}        
\newtheorem{remark}[thm]{Remark}   
\newtheorem{ex}[thm]{Example}        
\newtheorem{rmk}[thm]{Remark}
\newtheorem{notation}[thm]{Notation}

\newtheorem{notterm}[thm]{Notation and Terminology}

\begin{document}

\pagenumbering{arabic}

\title[Arithmetic of the secant variety]{Arithmetic properties of the first secant variety to a projective variety}




\author[Peter Vermeire]{Peter Vermeire}

\address{Department of Mathematics, 214 Pearce, Central Michigan
University, Mount Pleasant MI 48859}

\email{verme1pj@cmich.edu}


\date{\today}

\begin{abstract} 
Under an explicit positivity condition, we show the first secant variety of a linearly normal smooth variety is projectively normal, give results on the regularity of the ideal of the secant variety, and give conditions on the variety that are equivalent to the secant variety being arithmetically Cohen-Macaulay.  Under this same condition, we then show that if $X$ satisfies $N_{p+2\dim(X)}$, then the secant variety satisfies $N_{3,p}$.
\end{abstract}

\maketitle

\section{Introduction}

We work throughout over an algebraically closed field of characteristic zero.  Secant varieties are a classical subject, though the majority of work done involves determining the dimensions of secant varieties to well-known varieties.  Perhaps the two most well-known results in this direction are the solution by Alexander and Hirschowitz (completed in \cite{ah}) of the Waring problem for homogeneous polynomials and the classification of the Severi varieties by Zak \cite{zak}.  

More recently there has been great interest, e.g. related to algebraic statistics and algebraic complexity, in determining the equations defining secant varieties (e.g. \cite{ar}, \cite{bcg}, \cite{BGL}, \cite{CGG05a}, \cite{CGG05b}, \cite{CGG07}, \cite{CGG}, \cite{unex}, \cite{cs}, \cite{gss}, \cite{kanev}, \cite{Lan06}, \cite{Lan08}, \cite{LM08}, \cite{LW07}, \cite{LW08}, \cite{LO},   \cite{LO2}, \cite{OB}, \cite{sidsul}, \cite{ss}).  In this work, we use the detailed geometric information concerning secant varieties developed by Bertram \cite{bertram}, Thaddeus \cite{thaddeus}, and the author \cite{vermeireflip1} to study not just the equations defining secant varieties, but the syzygies among those equations as well.  This program was carried out for smooth curves in \cite{vermeirecurves}.

Under an explicit positivity condition, we show that the first secant variety $\Sigma$ to a smooth projective variety $X^d\subset\P^n$ is projectively normal (Theorem~\ref{projnor}) and that $\I_{\Sigma}$ is $(2d+3)$-regular (Corollary~\ref{firstreg}), directly extending results of \cite{vermeiresecreg} for smooth curves.  We also obtain simple conditions on the intrinsic geometry of $X$ which are equivalent to the condition that $\Sigma$ is arithmetically Cohen-Macaulay (Theorem~\ref{acm}), extending results of \cite{sidver} for curves.  We then show (Theorem~\ref{forgeneral}) that if $X^d$ satisfies $N_{p+2d}$, then $\Sigma$ satisfies $N_{3,p}$ (see Corollares~\ref{list} and~\ref{list2}  for a list of specific examples).


\begin{notterm}\label{maps}
Recall that an embedding $X^d\subset\P^n$ is \textbf{\boldmath $r$-very ample} if every subscheme of length $r+1$ spans a $\P^{r}\subset\P^n$, and that \textbf{\boldmath $X$ satisfies $N_{k,p}$} if the ideal of $X$ is generated in degree $k$ and the syzygies among the generators are linear for $p-1$ steps \cite{EGHP}.  It is immediate that if an embedding is $3$-very ample then $\operatorname{dim}(\Sigma)=2d+1$.

Under the hypotheses that $X\subset \P^n$ is a smooth variety such that the embedding is $3$-very ample and satisfies $N_{2,2}$, the reader should keep in mind the following morphisms \cite{vermeireflip1}
\begin{center}
{\begin{minipage}{1.5in}
\diagram 
& & \H \\
& Z \cong \operatorname{Bl}_{\Delta}(X \times X) \urto^{d=\varphi|_Z} \dlto^{\pi_2} \dto^{\pi_1 = \pi|_Z} \rto^{\hspace{.5in}i} & \wts \uto_{\varphi} \dto^{\pi}\\
X & X  \rto  &\Sigma   
\enddiagram
\end{minipage}}
\end{center}
where 
\begin{itemize}
\item $\pi$ is the blow up of $\Sigma$ along $X$
\item  $i$ is the inclusion of the exceptional divisor of the blow-up
\item $d$ is the double cover, $\pi_i$ are the projections
\item $\varphi$ is the morphism induced by the linear system $|2H-E|$ which gives $\wts$ the structure of a $\P^1$-bundle over $\H$; note in particular that $\wts$ is smooth.  
\end{itemize}
Note that we make extensive use of the rank $2$ vector bundle $\E_L=\varphi_*\mathcal{O}(H)=d_*\left(L\boxtimes\mathcal{O}\right)$, and note that for $i\geq1$, $R^i\pi_*\mathcal{O}_{\wts}=H^i(X,\mathcal{O}_X)\tensor\mathcal{O}_X$ (this is shown in \cite[Proposition 9]{vermeiresecreg} for curves, but the same proof works in the general case).
\end{notterm}

The positivity condition we will invoke is:

\begin{notation}\label{starstar}
For $p\geq 0$, we say $X\subset\P^n$ satisfies $N^{\Sigma}_p$ if
\begin{enumerate}
\item the embedding of $X$ is $3$-very ample and satisfies $N_{2,p}$; and
\item $H^i(\wts,\mathcal{O}_{\wts}(bH-E))=0$ for $i,b\geq1$.
\end{enumerate} 
\end{notation}

We devote the next section to the study of $N^{\Sigma}_p$.

\begin{remark}
Note that in Notation and Terminology~\ref{maps}, the morphism $\varphi$ induced by $|2H-E|$ embeds $\H\subset\P^s=\P(\Gamma(\I_X(2))$.  Writing $\mathcal{O}_{\H}(1)=\varphi^*\mathcal{O}_{\P^s}(1)$, it will be shown in the proof of Proposition~\ref{vanishonsec} that if $H^i(\H,\mathcal{O}(r))=H^i(\H,\E_L(r))=0$ for $i,r\geq1$, then the vanishing condition in Notation~\ref{starstar} is satisfied.  Thus the vanishing condition is a reasonable positivity condition.
\qed
\end{remark}

\section{Condition $\NS$}

For curves, verification of $\NS$ is straightforward.
\begin{prop}\label{starstarforcurves}
Let $X\subset\P^n$ be a smooth curve satisfying $N_{p}$, $p\geq2$, with $L=\mathcal{O}_X(1)$ non-special.  Then $L$ satisfies $\NS$.
\end{prop}

\begin{proof}
We need to show $H^i(\wts,\mathcal{O}_{\wts}(bH-E))=0$ for $i,b\geq1$.

Because $X$ is projectively normal we have $H^i(\wtp,\mathcal{O}_{\wtp}(bH-E))=0$ for $i,b\geq1$.  Thus $H^i(\wts,\mathcal{O}_{\wts}(bH-E))=H^{i+1}(\wtp,\mathcal{O}_{\wtp}(bH-E)\tensor\I_{\wts})$.  By \cite[2.4(6)]{sidver}, we know that $H^{i+1}(\wtp,\mathcal{O}_{\wtp}(bH-E)\tensor\I_{\wts})=H^{i+1}(\P^n,\I_{\Sigma}(b))$ (see also Lemma~\ref{change} where this is shown to be true in all dimensions).  

Now, for $i\geq1$, the arguments in \cite{vermeiresecreg} and in \cite{sidver} go through under the stated hypotheses to give $H^{i+1}(\P^n,\I_{\Sigma}(b))=0$ for $b\geq1$.  The extra hypothesis used in those papers (namely, that $\deg(L)\geq 2g+3$) is needed only to show $H^{1}(\P^n,\I_{\Sigma}(b))=0$ for $b\geq1$.
\qed
\end{proof}

Verifying condition $\NS$ in the general case takes somewhat more work, but the end results are reasonable.
We first need a computation which will be used in both Proposition~\ref{vanishonsec} and in Theorem~\ref{forgeneral}.
\begin{lemma}\label{det}
Let $X$ be a smooth variety embedded by a $3$-very ample line bundle $L$ satisfying $N_{2,2}$.  Then $d^*\wedge^2\E_L=L\boxtimes L(-E_{\Delta})$.
\end{lemma}

\begin{proof}
Consider the sequence on $\wts$:
$$\ses{\mathcal{O}_{\wts}(-E)}{\mathcal{O}_{\wts}}{\mathcal{O}_Z}$$
As $R^0\varphi_*\mathcal{O}_{\wts}(-E)=0$, pushing down to $\H$ we have (\cite[3.10]{sidver})
$$\ses{\mathcal{O}_{\H}}{\mathcal{O}_{\H}\oplus M}{R^1\varphi_*\mathcal{O}_{\wts}(-E)}$$
where $d^*M=\mathcal{O}_{Z}(-E_{\Delta})$.

Thus $R^1\varphi_*\mathcal{O}_{\wts}(-E)=M$.  However, we know by \cite[5.1.2]{conrad} that $$\left(R^1\varphi_*\mathcal{O}_{\wts}(-E)\right)^*=R^0\varphi_*\left(\omega_{\wts/\H}\tensor\mathcal{O}_{\wts}(E)\right)$$ where $\omega_{\wts/\H}=\varphi^*\wedge^2\E_L(-2H)$ \cite[Ex.III.8.4b]{hart}.  Thus we have
$$M^*=\wedge^2\E_L\tensor\mathcal{O}_{\H}(-1)$$
and so $\varphi^*\wedge^2\E_L=\mathcal{O}_{\wts}(2H-E)\tensor\varphi^*M^*$.  Restricting (pulling back) this equality to $Z$ and noting (\cite[3.6]{vermeireidealreg}) that $\mathcal{O}_{Z}(2H-E)=L\boxtimes L(-2E_{\Delta})$, we have $d^*\wedge^2\E_L=L\boxtimes L(-E_{\Delta})$.
\qed
\end{proof}

We now interpret the vanishing condition in the definition of $\NS$ in terms of $X$.

\begin{prop}\label{vanishonsec}
Let $X\subset\P^n$ be a smooth variety embedded by a $3$-very ample line bundle $L$ satisfying $N_{2,2}$ such that $H^i(X\times X,L^{r+s}\boxtimes L^r\tensor \I_{\Delta}^{q})=0$ for $i,r\geq1$, $s\geq0$, $0\leq q\leq2r$.  Then $H^i(\wts,\mathcal{O}_{\wts}(bH-E))=0$ for $i,b\geq1$.
\end{prop}

\begin{proof}
Suppose $b=2r$ is even.  We know by the proof of \cite[3.6]{vermeireidealreg} that $\mathcal{O}_Z(bH-E)=L^{b-1}\boxtimes L\tensor\mathcal{O}(-2\Delta)$; thus
$$H^i(Z,\mathcal{O}_{Z}(bH-rE))=H^i(X\times X,L^r\boxtimes L^r\tensor \I_{\Delta}^{2r})=0$$  Because $\mathcal{O}_{\wts}(bH-rE)=\varphi^*\mathcal{O}_{\H}(r)$, we know $d_*\mathcal{O}_Z(bH-rE)=\mathcal{O}_{\H}(r)\tensor(\mathcal{O}\oplus M)$ for some line bundle $M$, and hence we know that $H^i(\H,\mathcal{O}_{\H}(r))=0$, but this says that 
$H^i(\wts,\mathcal{O}_{\wts}(bH-rE))=0$. From the sequences $$\ses{\mathcal{O}_{\wts}(bH-(k+1)E)}{\mathcal{O}_{\wts}(bH-kE)}{\mathcal{O}_{Z}(bH-kE)}$$
for $k+1\leq r$ we see that $H^i(\wts,\mathcal{O}_{\wts}(bH-E))=0$, as the cohomology of the rightmost terms vanishes by hypothesis since $H^i(Z,\mathcal{O}_{Z}(bH-kE))=H^i(X\times X,L^{b-k}\boxtimes L^k\tensor \I_{\Delta}^{2k})=0$.  

Now, suppose that $b=2r+1$ is odd.  As in the previous paragraph, we have $\mathcal{O}_{\wts}((b-1)H-rE)=\varphi^*\mathcal{O}_{\H}(r)$, thus we see that $\varphi_*\mathcal{O}_{\wts}(bH-rE)=\mathcal{O}_{\H}(r)\tensor\varphi_*\mathcal{O}_{\wts}(H)=\mathcal{O}_{\H}(r)\tensor \E$.  It is therefore enough to show that $H^i(\H,\mathcal{O}_{\H}(r)\tensor \E)=0$, and then repeating the same argument as above gives $H^i(\wts,\mathcal{O}_{\wts}(bH-E))=0$.

We have the sequence on $Z$
$$\ses{K}{d^*\E_L}{L\boxtimes\mathcal{O}}$$
where $K=d^*\wedge^2\E_L\tensor \left(L^*\boxtimes\mathcal{O}\right)=\mathcal{O}\boxtimes L(-E_{\Delta})$ by Lemma~\ref{det}.  As in the proof of Lemma~\ref{det}, we have $d_*d^*\E_L=\E_L\oplus(\E_L\tensor M)$, thus $$d_*\left(\mathcal{O}_Z(2rH-rE)\tensor d^*\left(\E_L\tensor M^*\right)\right)=\E_L\tensor M^*(r)\oplus \E_L(r)$$
Thus it suffices to show $H^i(Z,\mathcal{O}_Z(2rH-rE)\tensor d^*\left(\E_L\tensor M^*\right))=0$.  However, we have 
$$K\tensor \mathcal{O}_Z(2rH-rE)\tensor d^*M^*=L^r\boxtimes L^{r+1}(-2rE_{\Delta})$$
and
$$L\boxtimes\mathcal{O} \tensor \mathcal{O}_Z(2rH-rE)\tensor d^*M^*=L^{r+1}\boxtimes L^r((-2r+1)E_{\Delta})$$
and so the cohomology of each vanishes by hypothesis.
\qed
\end{proof}

Fortunately, the vanishing in Proposition~\ref{vanishonsec} is not too difficult to understand.

\begin{prop}\label{uniform}
Let $X$ be a smooth variety of dimension $d$, $M$ a very ample line bundle.  Choose $k$ so that $k\geq d+3$ and so that $M^{k-d-1}\tensor\omega_X^*$ is big and nef.  Letting $L=M^k$, we have $$H^i(X\times X,L^{r+s}\boxtimes L^r\tensor \I_{\Delta}^{q})=0$$
for $i,r\geq1$, $s\geq0$, $0\leq q\leq2r$.  
\end{prop}

\begin{proof}
Note as above that $H^i(X\times X,L^{r+s}\boxtimes L^r\tensor \I_{\Delta}^{2r})=H^i(Z,L^{r+s}\boxtimes L^r\tensor \mathcal{O}(-2rE_{\Delta}))$, where $E_{\Delta}\rightarrow\Delta$ is the exceptional divisor of the blow-up.  Note further that $K_Z=K_X\boxtimes K_X\tensor \mathcal{O}((\dim X-1)E_{\Delta})$.

Assume first that $r\geq2$.
Then $$L^{r+s}\boxtimes L^r\tensor \mathcal{O}(-2rE_{\Delta})=K_Z\tensor (L^{r+s}-K_X)\boxtimes(L^r-K_X)\tensor \mathcal{O}((-d+1-q)E_{\Delta})$$
but this is $K_Z+B$ where
$$B=\left[(L-K_X)\boxtimes(L-K_X)\right]\tensor \left[L^{r+s-1}\boxtimes L^{r-1}\tensor\mathcal{O}\left((-d+1-q)E_{\Delta}\right)\right]$$

Because $M^k-K_X$ is ample, $(L-K_X)\boxtimes(L-K_X)$ is ample.  We are thus left to show that 
$$M^{k(r+s-1)}\boxtimes M^{k(r-1)}\tensor\mathcal{O}\left((-d+1-q)E_{\Delta}\right)$$
is globally generated.  However, as $k\geq d+3$, we have $k(r-1)\geq d-1+2r$ and so $M^{k(r+s-1)}\boxtimes M^{k(r-1)}\tensor\mathcal{O}\left((-d+1-2r)E_{\Delta}\right)$ is globally generated by \cite[3.1]{bel}.  Thus $B$ is big and nef and so vanishing follows from Kawamata-Viehweg vanishing \cite{kaw},\cite{vie}.

Now let $r=1$.  Then 
$$L^{1+s}\boxtimes L\tensor \mathcal{O}(-2E_{\Delta})=K_Z\tensor (M^{k+ks}-K_X)\boxtimes(M^{k}-K_X)\tensor \mathcal{O}((-d+1-q)E_{\Delta})$$
but this is $K_Z+B$ where
$$B=\left[(M^{k-d-1}-K_X)\boxtimes (M^{k-d-1}-K_X)\right]\tensor \left[M^{ks+d+1}\boxtimes M^{d+1}\right]\tensor \mathcal{O}((-d+1-q)E_{\Delta})$$
As above, $B$ is big and nef.
\qed
\end{proof}

\begin{rmk}\label{adjoint}
There are numerous ways to rearrange the terms in Proposition~\ref{uniform} to produce the desired vanishing.  

For example, a similar argument shows that if $M$ is very ample, $\omega_X\tensor M$ is big and nef, and $B$ is nef, then letting $L=\omega_X\tensor M^k\tensor B$ gives the vanishing for $k\geq d+2$ (Cf. \cite[Theorem 1]{el}).
If, further, $B$ is also big, then letting $L=\omega_X\tensor M^k\tensor B$ gives the vanishing for $k\geq d+1$.
\end{rmk}

\begin{rmk}\label{fano}
In Proposition~\ref{uniform}, if $\omega_X^*$ is big and nef (e.g. $X$ is Fano) then a slight revision of the argument shows it is enough to take $L=M^k$ for $k\geq d+1$.
\end{rmk}

\begin{rmk}\label{gaussianetal} 
Note that the vanishing condition in Proposition~\ref{vanishonsec} is intimately related to the surjectivity of the higher-order Gauss-Wahl maps as defined in \cite{wahl}.  Note in particular part (7) of Corollary~\ref{list}.
\end{rmk}

\begin{cor}\label{list}
The following embedded varieties satisfy $\NS$, $p\geq2$:
\begin{enumerate}
\item $X$ is a non-special smooth curve satisfying $N_p$ (Proposition~\ref{starstarforcurves}).
\item $X$ is a smooth variety embedded by a sufficiently high power of an ample line bundle.
\item $X^d\neq\P^d$ is a smooth variety embedded by $L=K_X\tensor M^{d+k}$, $k\geq p$, where $M$ is very ample and $K_X\tensor M$ is ample.
\item $X^d\neq\P^d$ is a smooth Fano variety embedded by $L=(-K_X)^{r}$ where $r\geq d+p-1$.
\item $X^d$ is an abelian variety embedded by $L^k$, where $L$ is ample and $k\geq2d+4$.
\item $X^d$ is a smooth projective toric variety embedded by $L^k$, where $L$ is ample, $L^{k-d-1}\tensor\omega_X^*$ is ample, and $k\geq \max\{d+3, d+p-1\}$.

\item $X=G/P$ where $G=\operatorname{SL}(V)$, $P$ is a parabolic subgroup, and $L=M^r$ where $M$ is a very ample line bundle such that the embedding by $L$ is $3$-very ample and $r\geq p$.
\item $X=v_{d_1,\ldots,d_r}(\P^{n_1}\times\cdots\times\P^{n_r})\subset\P^N$ where $d_i\geq p\geq3$.
\end{enumerate}
\end{cor}

\begin{proof}
For part (2), we note that a sufficiently high power of an ample line bundle satisfies $N_p$ by \cite{mgreen},\cite{inamdar}.  Further, the vanishing in Proposition~\ref{vanishonsec} is easily seen to hold for sufficiently high powers as well.

For part (3), it is shown in \cite[3.1]{el} that $L=K_X\tensor M^{d+k}$ satisfies $N_{k}$, $k\geq0$.  The result now follows from Remark~\ref{adjoint}.

Part (4) follows as in part (3) together with Remark~\ref{fano}.

For part (5), it is shown in \cite{bs} that $L^k$ is $(k-2)$-very ample and it is shown in \cite{par}, \cite{pp} that $L^k$ satisfies $N_{k-3}$.  It is shown in \cite[Theorem C]{parwahl} that $H^i(X\times X,(L^k)^{r+s}\boxtimes (L^k)^r\tensor \I_{\Delta}^{2r})=0$ for $i,r\geq1$, $s\geq0$ and $k\geq6$.

For (6), it is shown in \cite{hss} that $X$ satisfies $N_{p}$.  The result now follows by Proposition~\ref{uniform}.

For (7), it is shown in \cite{man} that if $X=G/P$ where $G=\operatorname{SL}(V)$, $P$ is a parabolic subgroup, and $L$ is a very ample line bundle, then the embedding by $L^p$ satisfies $N_p$.  By \cite[2.5]{kumar} and \cite[6.5]{wahl2} we know that $H^i(X\times X,L^{r+s}\boxtimes L^r\tensor \I_{\Delta}^{2r})=0$ for $i,r\geq1$, $s\geq0$ as long as $L=M^k$, $k\geq2$.

For (8), it is shown in \cite{hss} that $X$ satisfies $N_p$, and again by \cite[2.5]{kumar} and \cite[6.5]{wahl2} we are done.

\qed
\end{proof}

\section{Projective Normality}

\begin{thm}\label{projnor}
If $X^d\subset\P^n$ is smooth, projectively normal, and satisfies $N_{2d}^{\Sigma}$, then $\Sigma$ is projectively normal.
\end{thm}

\begin{proof}
By \cite[2.2]{vermeiresing} $\Sigma$ is normal and by \cite[Remark 13]{vermeiresecreg} $\Sigma\subset\P^n$ is linearly normal.

We use the fact that $H^1(\P^n,\I_{\Sigma}(1))=0$ and the standard diagram
\begin{center}
{\begin{minipage}{1.5in}
\diagram
 & & &  0\dto   &\\
 & & &  \I_{\Sigma}(k+1)\dto   &\\
0\rto&\mathcal{O}mega^1_{\P^n}(k+1)\dto\rto&\Gamma(\P^n,\mathcal{O}(1))\tensor\mathcal{O}_{\P^n}(k)\dto\rto&\mathcal{O}_{\P^n}(k+1)\dto\rto&0\\
0\rto&\mathcal{O}mega^1_{\P^n}\tensor\mathcal{O}_{\Sigma}(k+1)\dto\rto&\Gamma(\Sigma,\mathcal{O}(1))\tensor\mathcal{O}_{\Sigma}(k)\dto\rto&\mathcal{O}_{\Sigma}(k+1)\dto\rto&0\\
  & 0 & 0 & 0 &
\enddiagram
\end{minipage}}
\end{center}
By induction on $k\geq1$, we see that if $H^1(\Sigma,\mathcal{O}mega^1_{\P^n}\tensor\mathcal{O}_{\Sigma}(k+1))=0$ then $H^1(\P^n,\I_{\Sigma}(k+1))=0$.  We will show below (Theorem~\ref{forgeneral}) that $H^1(\Sigma,\mathcal{O}mega^1_{\P^n}\tensor\mathcal{O}_{\Sigma}(k+1))=0$ for $k\geq2$ as a consequence of a more general approach studying the syzygies of $\I_{\Sigma}$.  It will thus be sufficient to show that $H^1(\P^n,\I_{\Sigma}(2))=0$.

Consider the morphism $d:Z\rightarrow\H$; we write $d_*(L\boxtimes \mathcal{O})=\E$.  Pushing the sequence
$$\ses{d^*M_{\E}\tensor (L\boxtimes\mathcal{O})}{M_{L\boxtimes\mathcal{O}}\tensor (L\boxtimes\mathcal{O})}{L\boxtimes L(-E_{\Delta})}$$
down to $\H$ yields
$$\ses{M_{\E}\tensor\E}{(\varphi_*\pi^*\mathcal{O}mega^1_{\P^n}\tensor\mathcal{O}_{\Sigma}(2))\oplus\mathcal{O}_{\H}(1)}{\wedge^2\E\oplus\mathcal{O}_{\H}(1)}.$$

From the sequence on $\wts$
$$\ses{\pi^*\mathcal{O}mega^1_{\P^n}\tensor\mathcal{O}_{\wts}(2H-E)}{\Gamma(\Sigma,\mathcal{O}(1))\tensor\mathcal{O}_{\wts}(H-E)}{\mathcal{O}_{\wts}(2H-E)}$$
and because the restriction of $\mathcal{O}_{\wts}(H-E)$ to a fiber of the $\P^1$-bundle $\varphi:\wts\rightarrow\H$ is $\mathcal{O}(-1)$, we immediately see that $\varphi_*\left[\pi^*\mathcal{O}mega^1_{\P^n}\tensor\mathcal{O}_{\wts}(2H-E)\right]=0$ and $R^1\varphi_*\left[\pi^*\mathcal{O}mega^1_{\P^n}\tensor\mathcal{O}_{\wts}(2H-E)\right]=\mathcal{O}_{\H}(1)$.

Putting these together, consider the sequence on $\wts$
$$\ses{\pi^*\mathcal{O}mega^1_{\P^n}\tensor\mathcal{O}_{\wts}(2H-E)}{\pi^*\mathcal{O}mega^1_{\P^n}\tensor\mathcal{O}_{\wts}(2H)}{\pi^*\mathcal{O}mega^1_{\P^n}\tensor\mathcal{O}_{Z}(2H)}$$
Applying $\varphi_*$ yields
$$\begin{array}{llcllll} 
0&\rightarrow & 0&\rightarrow &\varphi_*\pi^*\mathcal{O}mega^1_{\P^n}\tensor\mathcal{O}_{\wts}(2H)&\rightarrow &(\varphi_*\pi^*\mathcal{O}mega^1_{\P^n}\tensor\mathcal{O}_{\Sigma}(2))\oplus\mathcal{O}_{\H}(1) \\
&\rightarrow & \mathcal{O}_{\H}(1) &\rightarrow & 0  & &
\end{array}$$
and so $H^i(Z,\pi^*\mathcal{O}mega^1_{\P^n}\tensor\mathcal{O}_{Z}(2H))$ splits as a direct sum; in particular, $$H^1(\wts,\pi^*\mathcal{O}mega^1_{\P^n}\tensor\mathcal{O}_{\wts}(2H))\rightarrow H^1(Z,\pi^*\mathcal{O}mega^1_{\P^n}\tensor\mathcal{O}_{Z}(2H))$$ is an injection.  However, by the K\"unneth formula $H^1(Z,\pi^*\mathcal{O}mega^1_{\P^n}\tensor\mathcal{O}_{Z}(2H))=H^0(X,\mathcal{O}mega^1_{\P^n}\tensor\mathcal{O}_{X}(2))\tensor H^1(X,\mathcal{O}_X)$, but this is precisely $H^0(\Sigma,R^1\pi_*\pi^*\mathcal{O}mega^1_{\P^n}\tensor\mathcal{O}_{\wts}(2H))$, hence $H^1(\Sigma,\pi_*\pi^*\mathcal{O}mega^1_{\P^n}\tensor\mathcal{O}_{\wts}(2H))=H^1(\Sigma,\mathcal{O}mega^1_{\P^n}\tensor\mathcal{O}_{\Sigma}(2))=0$.

\qed
\end{proof}

\begin{cor}
In all the examples of Remark~\ref{list}, $\Sigma$ is projectively normal for $p\geq 2d$.
\qed
\end{cor}

\section{Regularity and Cohen-Macaulayness}

\begin{lemma}\label{change}
Suppose $X\subset\P^n$ is a $3$-very ample embedding of a smooth projective variety satisfying $N_{2,2}$.  Then $H^i(\P^n,\I_{\Sigma}(k))=H^i(B_2,\mathcal{O}(kH-E_1-E_2))$.
\end{lemma}

\begin{proof}
Consider the sequence
$$\ses{\mathcal{O}_{B_2}(kH-E_1-E_2)}{\mathcal{O}_{B_2}(kH-E_2)}{\mathcal{O}_{E_1}(kH-E_2)}$$
We know that $R^i\pi_*\mathcal{O}_{E_1}(kH-E_2)=0$ for $i=0,1$, and that $R^i\pi_*\mathcal{O}_{E_1}(kH-E_2)=H^{i-1}(X,\mathcal{O}_X)\tensor\mathcal{O}_X(k)$ otherwise.  From the sequence
$$\ses{\mathcal{O}_{B_2}(kH-E_2)}{\mathcal{O}_{B_2}(kH)}{\mathcal{O}_{\wts}(kH)}$$
we see that $R^{i}\pi_*\mathcal{O}_{B_2}(kH-E_2)=R^{i-1}\pi_*\mathcal{O}_{\wts}(kH)=H^{i-1}(X,\mathcal{O}_X)\tensor\mathcal{O}_X(k)$ for $i\geq2$.  Thus a local computation gives $R^i\pi_*\mathcal{O}_{B_2}(kH-E_1-E_2)=0$ for $i\geq1$, and so $H^i(B_2,\mathcal{O}_{B_2}(kH-E_1-E_2))=H^i(\P^n,R^0\pi_*\mathcal{O}_{B_2}(kH-E_1-E_2))=H^i(\P^n,\I_{\Sigma}(k))$.
\qed
\end{proof}

\begin{prop}
Suppose $X\subset\P^n$ is projectively normal and satisfies $N_{2d}^{\Sigma}$, and that $H^i(X,\mathcal{O}_X(r))=0$ for $i,r\geq1$.  Then $H^i(\P^n,\I_{\Sigma}(k))=0$ for $i,k\geq1$.
\end{prop}

\begin{proof}
We use the condition found in Lemma~\ref{change}.  We already have this for $i=1$.  For $k=1$, $i>1$, consider the sequence
$$\ses{\mathcal{O}_{B_2}(H-E_1-E_2)}{\mathcal{O}_{B_2}(H-E_1)}{\mathcal{O}_{\wts}(H-E_1)}$$
As $H^i(\wts,\mathcal{O}_{\wts}(H-E_1))=0$, we have $H^i(\P^n,\I_{\Sigma}(1))=H^i(\P^n,\I_X(1))=H^{i-1}(X,\mathcal{O}_X(1))=0$ for $i\geq2$.

We now have our result for $i=1$ and for $k=1$.  This gives $H^i(\Sigma,\mathcal{O}_{\Sigma}(1))=0$ for $i\geq1$.  Thus, by induction on $k$ it suffices to show that $H^i(\Sigma,\mathcal{O}mega^1_{\P^n}\tensor\mathcal{O}_{\Sigma}(k))=0$ for $i,k\geq2$ just as in the proof of Theorem~\ref{projnor}.   Again, in Theorem~\ref{forgeneral} we show that $H^i(\Sigma, \mathcal{O}mega^1_{\P^n}\tensor\mathcal{O}_{\Sigma}(k))=0$ for $i\geq 2$ and $k\geq3$.  Thus we will be left to show  $H^i(\Sigma,\mathcal{O}mega^1_{\P^n}\tensor\mathcal{O}_{\Sigma}(2))=0$ for $i\geq 2$.  Equivalently, we may show $H^i(\Sigma,\mathcal{O}_{\Sigma}(2))=0$ for $i\geq1$.  From the sequence
$$\ses{\mathcal{O}_{B_2}(2H-E_1-E_2)}{\mathcal{O}_{B_2}(2H-E_1)}{\mathcal{O}_{\wts}(2H-E_1)}$$     
it suffices to show $H^i(\wts,\mathcal{O}_{\wts}(2H-E_1))=0$ for $i\geq1$.  Consider the line bundle $L\boxtimes L(-E_{\Delta})$ on $Z$.  We know that $H^i(Z,L\boxtimes L(-E_{\Delta}))=H^i(X,\mathcal{O}mega^1_{\P^n}(2))=0$ for $i\geq1$ by hypothesis.  However, $d_*\left(L\boxtimes L(-E_{\Delta})\right)=\wedge^2\E\oplus\mathcal{O}_{\H}(1)$, hence $H^i(\H,\mathcal{O}_{\H}(1))=H^i(\wts,\mathcal{O}_{\wts}(2H-E_1))=0$ for $i\geq1$.
\qed
\end{proof}

\begin{cor}\label{firstreg}
Suppose $X\subset\P^n$ is smooth, projectively normal, and satisfies $N_{2d}^{\Sigma}$, and that $H^i(X,\mathcal{O}_X(r))=0$ for $i,r\geq1$.  Then $\I_{\Sigma}$ is $(2d+3)$-regular.
\qed
\end{cor}

\begin{prop}\label{partial}
If $X\subset\P^n$ is amooth, projectively normal, and satisfies $N_2^{\Sigma}$, then $H^i(\P^n,\I_{\Sigma}(k))=0$ for $k<0$, $0\leq i\leq d+2$.
\end{prop}

\begin{proof}
This is obvious for $i=0$.  For $i>0$, we show $H^{i-1}(\Sigma,\mathcal{O}_{\Sigma}(k))=0$.  By Kawamata-Viehweg, we know $H^{i-1}(\wts,\mathcal{O}_{\wts}(kH))=0$ for $1\leq i\leq 2d+1$.

For $0\leq j\leq \operatorname{min}\{i-2,d-1\}$, we know that $H^j(\Sigma,R^{i-j-1}\pi_*\mathcal{O}_{\wts}(kH))=H^{i-j-1}(X,\mathcal{O}_X)\tensor H^j(X,\mathcal{O}_X(k))=0$ since $k<0$.  Thus for $j=i-1\leq d-1$, we have
\begin{eqnarray*}
H^{i-1}(\Sigma,\mathcal{O}_{\Sigma}(k))&=&H^{i-1}(\Sigma,R^0\pi_*\mathcal{O}_{\wts}(kH))\\
&=&H^{i-1}(\wts,\mathcal{O}_{\wts}(kH))\\
&=&0
\end{eqnarray*}
and hence $H^i(\P^n,\I_{\Sigma}(k))=0$ for $k<0$ and $0\leq i\leq d+1$.  

To show $H^{d+1}(\Sigma,\mathcal{O}_{\Sigma}(k))=0$ for $k<0$, note that $$H^j(\Sigma,R^{d+1-j}\pi_*\mathcal{O}_{\wts}(kH))=H^{d+1-j}(X,\mathcal{O}_X)\tensor H^j(X,\mathcal{O}_X(k))=0$$ for $j<d$.  Thus $E_2^{0,d+1}=E_{\infty}^{0,d+1}=0$.  Looking at the $E_*^{d+1,0}$ terms, we have the complexes
$$E_i^{d+1-i,i-1}\stackrel{d_i}{\rightarrow}E_i^{d+1,0}\rightarrow 0$$
but we just proved that $E_i^{d+1-i,i-1}=E_2^{d+1-i,i-1}=0$, and hence $E_2^{d+1,0}=E_{\infty}^{d+1,0}$.  Now by \cite[6.1(1)]{sidver}, we have $E_2^{d+1,0}=0$.
\qed
\end{proof}

\begin{cor}\label{improvement}
Let $X\subset\P^n$ be a smooth, non-special curve satisfying $N_2$.  Then $\Sigma$ is ACM and $\I_{\Sigma}$ is $5$-regular.
\qed
\end{cor}

\begin{rmk}
Corollary~\ref{improvement} was proved for embeddings of degree at least $2g+3$ in \cite{sidver} and \cite{vermeiresecreg}.
\end{rmk}

\begin{prop}
Suppose $X^d\subset\P^n$ is a smooth variety satisfying $N_2^{\Sigma}$.  If $d\geq2$ and $H^i(X,\mathcal{O}_X)\neq0$ for some $i\geq1$, then $\Sigma$ is not ACM.
\end{prop}

\begin{proof}
Suppose $H^i(X,\mathcal{O}_X)\neq0$ and consider the spectral sequence with $E_2^{a,b}=H^a(\Sigma,R^b\pi_*\mathcal{O}_{\wts}(kH))$ for $k<0$.  It is straightforward to check that 
$E_{i+2}^{d,i}=E_{\infty}^{d,i}$ and that $E_{i+2}^{d+i+1,0}=E_{\infty}^{d+i+1,0}$; from the fact that $H^{j}(\wts,\mathcal{O}_{\wts}(k))=0$ for $j\leq 2d$, we know that $E_{\infty}^{d,i}=E_{\infty}^{d+i+1,0}=0$ for $i\leq d-1$.  Therefore, from the complex
$$0\rightarrow E_{i+1}^{d,i}\stackrel{d_{i+1}}{\rightarrow}E_{i+1}^{d+i+1,0}\rightarrow 0$$
we see that the nontrivial map is actually an isomorphism, hence we have
\begin{eqnarray*}
H^{d+i+2}(\P^n,\I_{\Sigma}(k))&=&H^{d+i+1}(\Sigma,\mathcal{O}_{\Sigma}(k))\\
&=&H^{d+i+1}(\Sigma,R^0\pi_*\mathcal{O}_{\wts}(kH))\\
&=&E_2^{d+i+1,0}\\
&=&E_{i+1}^{d+i+1,0}\\
&=&E_{i+1}^{d,i}\\
&=&E_{2}^{d,i}\\
&=&H^d(\Sigma,R^i\pi_*\mathcal{O}_{\wts}(kH))\\
&=&H^d(X,H^i(X,\mathcal{O}_X)\tensor\mathcal{O}_X(k))\\
&=&H^{i}(X,\mathcal{O}_X)\tensor H^d(X,\mathcal{O}_X(k))
\end{eqnarray*}
However, as $k<0$ we know that $H^d(X,\mathcal{O}_X(k))\neq0$ for all $k<<0$, thus $\Sigma$ is not ACM.
\qed
\end{proof}

\begin{cor}
Suppose $X^d\subset\P^n$ is a smooth variety satisfying $N_2^{\Sigma}$. 
If $H^j(X,\mathcal{O}_X)=0$ for $j>0$, then $H^i(\P^n,\I_{\Sigma}(k))=0$ for $k<0$, $0\leq i\leq 2d+1$.
\qed
\end{cor}

\begin{thm}\label{acm}
Suppose $X^d\subset\P^n$ is a smooth variety of dimension $d\geq2$.  Suppose $X\subset\P^n$ is projectively normal and satisfies $N_{2d}^{\Sigma}$, and that $H^i(X,\mathcal{O}_X(r))=0$ for $i,r\geq1$.  Then the following are equivalent:
\begin{enumerate}
\item $H^j(X,\mathcal{O}_X)=0$ for $j>0$.
\item $\Sigma$ is ACM.
\item $\Sigma$ has rational singularities.
\end{enumerate}
Further, if one of these conditions is satisfied, then $\I_{\Sigma}$ is $(2d+1)$-regular.
\end{thm}

\begin{proof}
Clearly, $H^i(\P^n,\I_{\Sigma})=0$ for $0\leq i\leq1$.  Thus we are left to show $H^i(\Sigma,\mathcal{O}_{\Sigma})=0$ for $1\leq i\leq2d$.  By \cite[3.10]{sidver} we have $H^i(Z,\mathcal{O}_Z)\cong H^i(\wts,\mathcal{O}_{\wts})\oplus H^{i+1}(\wts,\mathcal{O}_{\wts}(-E_1))$.  By hypothesis we have $H^i(Z,\mathcal{O}_Z)=0$ for $i\geq0$, hence $H^i(\wts,\mathcal{O}_{\wts})=0$ for $i\geq0$.  However, our hypothesis also implies that $R^i\pi_*\mathcal{O}_{\wts}=0$ for $i\geq1$, hence $H^i(\wts,\mathcal{O}_{\wts})=H^i(\Sigma,\mathcal{O}_{\Sigma})=0$. 
\qed
\end{proof}

\begin{remark}
\textit{Macaulay 2} \cite{M2} calculations performed by Jessica Sidman show that for $v_3(\P^2)$ and for $v_4(\P^2)$, $\Sigma$ is $5$-regular but not $4$-regular. 
\end{remark}

\section{Syzygies}

Having established the basic normality and regularity results, following \cite{mgreen} we turn our attention to defining equations and syzygies.

Our starting point is the familiar:

\begin{prop}\label{basic}
Let $X\subset\P^n$ be a smooth variety embedded by a line bundle $L$.  Then $\Sigma$ satisfies $\operatorname{N}_{3,p}$ if $H^1(\Sigma,\wedge^aM_L(b))=0$, $2\leq a\leq p+1$, $b\geq 2$.  

\end{prop}

\begin{proof}
Because $L$ also induces an embedding $\Sigma\subset\P^n$, we abuse notation and denote the associated vector bundle on $\Sigma$ by $M_L$.
Letting $F=\oplus\Gamma(\Sigma_1,\mathcal{O}_{\Sigma}(n))$ and
applying \cite[5.8]{eisenbud} to $\mathcal{O}_{\Sigma}$ gives the exact sequence:
\[
0 \to \Tor_{a-1}(F, k)_{a+b} \to H^1(\Sigma_1,\wedge^{a} M_L(b)) \to H^1(\Sigma,\wedge^{a}\mathcal{O}_{\P}^{r+1} \otimes \mathcal{O}_{\Sigma_1}(b))
\]
The vanishing in the hypothesis implies that $\Tor_1(F,k)_d=0$ for $d \geq k+1$, and hence that the first syzygies of $\mathcal{O}_{\Sigma}$, which are the generators of the ideal of $\Sigma$, are in degree $\leq k$.  The rest of the vanishings yield the analogous statements for higher syzygies.

\nopagebreak \hfill $\Box$ \par \medskip
\end{proof}

The remaining technical portion of the paper is devoted to reinterpreting the vanishings in Proposition~\ref{basic} in terms of vanishings on the Hilbert scheme $\H$, and then finally on $X$ itself.

\begin{prop}\label{prop: wts trans}
If $X$ is a smooth variety embedded by a $3$-very ample line bundle $L$ satisfying $N_{2,2}$, then $\Sigma$ satisfies $\operatorname{N}_{3,p}$ if $$H^1(\wts,\pi^*\wedge^aM_L(b))\rightarrow H^0(\Sigma,\wedge^aM_L(b)\tensor R^1\pi_*\mathcal{O}_{\wts})$$ is injective for $2\leq a\leq p+1$, $b\geq 2$.

\end{prop}

\begin{proof}
This follows immediately from the start of the 5-term sequence associated to the Leray-Serre spectral sequence:
$$0\rightarrow H^1(\Sigma,\wedge^aM_L(b))\rightarrow H^1(\wts,\pi^*\wedge^aM_L(b))\rightarrow H^0(\Sigma,\wedge^aM_L(b)\tensor R^1\pi_*\mathcal{O}_{\wts})$$
and Proposition~\ref{basic}.
\qed
\end{proof}


\begin{prop}\label{prop: wts van}
Let $X\subset\P^n$ be a smooth variety embedded by a line bundle $L$ satisfying $N_{p}^{\Sigma}$ with $H^i(X,L^k)=0$ for $i,k\geq1$.  Then $\Sigma$ satisfies $\operatorname{N}_{3,p}$ if $H^i(\wts,\pi^*\wedge^{a-1+i}M_L\tensor\mathcal{O}(2H-E))=0$ for $2\leq a\leq p+1$, $i\geq1$.


\end{prop}

\begin{proof}
We use Proposition~\ref{prop: wts trans}.  From the sequence on $\wts$
$$\ses{\pi^*\wedge^aM_L(bH-E)}{\pi^*\wedge^aM_L(bH)}{\pi^*\wedge^aM_L(bH)\tensor\mathcal{O}_Z}$$
we know \begin{eqnarray*}
H^1(Z,\pi^*\wedge^aM_L(bH)\tensor\mathcal{O}_{Z})&=&H^1\left(Z,\left(\wedge^aM_L\tensor L^b\right)\boxtimes \mathcal{O}_X\right)\\
&=&H^1\left(X\times X,\left(\wedge^aM_L\tensor L^b\right)\boxtimes \mathcal{O}_X\right)\\
&=&H^1(X,\mathcal{O}_X)\tensor H^0(X,\wedge^aM_L\tensor L^b).
\end{eqnarray*}
The first equality follows as the restriction of $\pi^*\wedge^aM_L(bH)$ to $Z$ is $\wedge^aM_L(bH) \boxtimes \mathcal{O}_X$, the second is standard, and for the third we use the K\"unneth formula together with the fact that $h^1(X, \wedge^aM_L \otimes L^b)=0$ as $X$ satisfies $N_{2,p}$. 

Thus $$h^1(\Sigma,\wedge^aM_L(b))=\operatorname{Rank}\left(H^1(\wts,\pi^*\wedge^aM_L(bH-E))\rightarrow H^1(\wts,\pi^*\wedge^aM_L(bH))\right)$$
and so by Proposition~\ref{prop: wts trans} it is enough to show that $H^1(\wts,\pi^*\wedge^aM_L\tensor\mathcal{O}(bH-E))=0$ for $2\leq a\leq p+1$, $b\geq2$.

From the sequence 
$$\ses{\pi^*\wedge^{a+1}M_L\tensor\mathcal{O}(bH-E)}{\wedge^{a+1}\Gamma\tensor\mathcal{O}(bH-E)}{\pi^*\wedge^aM_L\tensor\mathcal{O}((b+1)H-E)}$$
and the fact that $H^i(\wts,\mathcal{O}(bH-E))=0$, we see that $H^1(\wts,\pi^*\wedge^{a}M_L\tensor\mathcal{O}(bH-E))=H^{b-2}(\wts,\pi^*\wedge^{a+b-2}M_L\tensor\mathcal{O}(2H-E))$ for $b\geq2$.




\qed
\end{proof}

\begin{lemma}\label{downtohilb}
Let $X$ be a smooth variety embedded by a $3$-very ample line bundle $L$ satisfying $N_{2,2}$ and consider the morphism $\varphi:\wts\rightarrow \H\subset\P^s$ induced by the linear system $|2H-E|$.  Then $\varphi_*\wedge^aM_L=\wedge^aM_{\mathcal{E_L}}$, and hence $H^i(\wts,\pi^*\wedge^aM_L\tensor\mathcal{O}(2H-E))=H^i(\H,\wedge^aM_{\E_L}\tensor \mathcal{O}_{\H}(1))$.
\end{lemma}

\begin{proof}
Consider the diagram on $\wts$:
\begin{center}
{\begin{minipage}{1.5in}
\diagram
 &  &  & 0\dto & \\
 & 0\dto & 0\dto & K\dto & \\
0\rto &  \varphi^*M_{\E_L}\dto\rto & \Gamma(\H,\E_L)\otimes \mathcal{O}_{\wts}\dto\rto & \varphi^*\E_L\dto\rto & 0 \\
0\rto &  \pi^*M_{L}\dto\rto & \Gamma(X,L)\otimes \mathcal{O}_{\wts} \dto\rto & \pi^*L\dto\rto & 0 \\
 &  K\dto & 0 & 0 &  \\
 & 0 &  &  &
\enddiagram
\end{minipage}}
\end{center}
The vertical map in the middle is surjective as we have $\Gamma(\H, \E_L) = \Gamma(\wts,\mathcal{O}(H)) = \Gamma(X\times X,L\boxtimes\mathcal{O}) = \Gamma(X, L)$.  Therefore, surjectivity of the lower right horizontal map and commutativity of the diagram show that the righthand vertical map is surjective.  

Note that $R^i\varphi_*\varphi^*\E_L = \E_L \otimes R^i \varphi_*\mathcal{O}_{\wts}$ by the projection formula and that the higher direct image sheaves $R^i \varphi_*\mathcal{O}_{\wts}$ vanish as $\wts$ is a $\PP^1$-bundle over $\H.$ For the higher direct images, we have $R^i \varphi_*\pi^*L=0$ as the restriction of $L$ to a fiber of $\varphi$ is $\mathcal{O}(1)$ and hence the cohomology along the fibers vanishes.  From the rightmost column, we see $R^i\varphi_*K=0$. From the leftmost column, we have the sequence
$$\ses{\varphi^*\wedge^aM_{\E_L}}{\pi^*\wedge^aM_{L}}{\varphi^*\wedge^{a-1}M_{\E_L}\tensor K}$$
but as $R^i\varphi_*\left(K\tensor\varphi^*\wedge^{a-1}M_{\E_L}\right)=R^i\varphi_*K\tensor\wedge^{a-1}M_{\E_L}=0$, we have $\varphi_*\wedge^aM_L=\wedge^aM_{\mathcal{E}_L}$.
\qed
\end{proof}

Combining Proposition~\ref{prop: wts van} with Lemma~\ref{downtohilb} yields:

\begin{cor}\label{onhilb}
Let $X$ be a smooth variety embedded by a line bundle $L$ satisfying $N_{p}^{\Sigma}$ with $H^i(X,L^k)=0$ for $i,k\geq1$.
Then $\Sigma$ satisfies $\operatorname{N}_{3,p}$ if $$H^i(\H,\wedge^{a-1+i}M_{\E_L}\tensor\mathcal{O}(1))=0$$ for $2\leq a\leq p+1$, $i\geq1$.\qed
\end{cor}

\begin{thm}\label{forgeneral}
Let $X^d$ be a smooth variety embedded by a line bundle $L$ satisfying $N_{p+2d}^{\Sigma}$ with $H^i(X,L^k)=0$ for $i,k\geq1$.  Then $\Sigma$ satisfies $\operatorname{N}_{3,p}$
\end{thm}

\begin{proof}
Pushing the sequence 
$$\ses{d^*\wedge^aM_{\E}\tensor (L\boxtimes\mathcal{O})}{\wedge^aM_{L\boxtimes\mathcal{O}}\tensor (L\boxtimes\mathcal{O})}{d^*\wedge^{a-1}M_{\E}\tensor(L\boxtimes L(-E_{\Delta}))}$$
down to $\H$ yields
$$\ses{\wedge^2M_{\E}\tensor\E}{\left(\varphi_*\pi^*\wedge^aM_L\tensor L\right)\oplus\left(\wedge^{a-1}M_{\E}(1)\right)}{\left(\wedge^{a-1}M_{\E}\tensor\wedge^2\E\right)\oplus\left(\wedge^{a-1}M_{\E}(1)\right)}$$
where the non-trivial part of the sequence comes from twisting the diagram in Lemma~\ref{downtohilb} by $\pi^*L$ and pushing down to $\H$.

From the sequence on $\wts$
$$\ses{\pi^*\wedge^aM_L\tensor\mathcal{O}_{\wts}(H-E)}{\wedge^a\Gamma(\Sigma,\mathcal{O}(1))\tensor\mathcal{O}_{\wts}(H-E)}{\pi^*\wedge^aM_L\tensor\mathcal{O}_{\wts}(2H-E)}$$
and noting that the restriction of $\mathcal{O}_{\wts}(H-E)$ to a fiber of the $\P^1$-bundle $\varphi:\wts\rightarrow\H$ is $\mathcal{O}(-1)$, we immediately see that: 
\begin{itemize}
	\item $\varphi_*\left[\pi^*\wedge^aM_L\tensor\mathcal{O}_{\wts}(H-E)\right]=0$
	\item $R^1\varphi_*\left[\pi^*\wedge^aM_L\tensor\mathcal{O}_{\wts}(H-E)\right]=\wedge^{a-1}M_{\E}(1)$
	\item $R^i\varphi_*\left[\pi^*\wedge^aM_L\tensor\mathcal{O}_{\wts}(H-E)\right]=0$
\end{itemize}
for $i\geq2$.

Putting these together, consider the sequence on $\wts$
$$\ses{\pi^*\wedge^aM_L\tensor\mathcal{O}_{\wts}(H-E)}{\pi^*\wedge^aM_L\tensor\mathcal{O}_{\wts}(H)}{\wedge^aM_{L\boxtimes\mathcal{O}}\tensor (L\boxtimes\mathcal{O})}$$
Applying $\varphi_*$ yields
$$\begin{array}{llclcll} 
0&\rightarrow & 0&\rightarrow &\varphi_*\pi^*\wedge^aM_L\tensor\mathcal{O}_{\wts}(H)&\rightarrow &(\pi^*\wedge^aM_L\tensor\mathcal{O}_{\wts}(H))\oplus\left(\wedge^{a-1}M_{\E}(1)\right) \\
&\rightarrow & \wedge^{a-1}M_{\E}(1) &\rightarrow & 0  & &
\end{array}$$
By the assumption that $X$ satisfies $N_{p+2d}$, we know that $H^i(Z,\wedge^aM_{L\boxtimes\mathcal{O}}\tensor (L\boxtimes\mathcal{O}))=H^0(X,\wedge^aM_L\tensor L)\tensor H^i(X,\mathcal{O}_X)$ for $0\leq a\leq p+2d+1$.  However, for $i\geq1$ this is precisely $H^0(\Sigma,R^i\pi_*\pi^*\wedge^aM_L\tensor\mathcal{O}_{\wts}(H))=E_2^{0,i}$.  It is straightforward to check that $E_2^{0,i}=E_{\infty}^{0,i}$, and thus we have an injection of $H^i$ into $E_{\infty}^{0,i}$; however, $E_{\infty}^{0,i}$ is a quotient of $H^i$, hence this is an isomorphism.

Thus we have $H^i(\wts,\pi^*\wedge^aM_L\tensor\mathcal{O}_{\wts}(H))\cong H^i(Z,\wedge^aM_{L\boxtimes\mathcal{O}}\tensor (L\boxtimes\mathcal{O}))$ for $i\geq1$ and $0\leq a\leq p+2d+1$.  In particular, we have $H^i(\H,\wedge^{a-1}M_{\E}(1))=0$ for $i\geq1$ and $0\leq a\leq p+2d+1$.  Together with Corollary~\ref{onhilb} this completes the proof.
\qed
\end{proof}

As above, we have:
\begin{cor}\label{list2}
In all the examples of Remark~\ref{list}, $\Sigma$ satisfies $N_{3,p-2d}$.
\qed
\end{cor}



\begin{ex}\label{pn}
Let $X^d_k=v_k(\P^d)\subset\P^N$, $k\geq3$.  We know by \cite{birkenhake} that $X_k^2$ satisfies $N_{3k-3}$, and hence by Corollary~\ref{list} we have $\Sigma$ satisfies $N_{3,3k-7}$.  

It has been shown \cite{bcr} that $X_k^d$ satisfies $N_{k+1}$ for all $d$, hence $\Sigma$ at least satisfies $N_{3,k-2d}$.
It is conjectured in \cite{op} that for $d\geq2$, $k\geq3$ we have $X_k^d$ satisfies $N_{3k-3}$, which would imply that $\Sigma$ satisfies $N_{3,3k-3-2d}$.  
\end{ex}

\textit{Macaulay 2} \cite{M2} calculations performed by Jessica Sidman show that for $v_3(\P^2)$, $\Sigma$ satisfies $N_{3,4}$ and for $v_4(\P^2)$, $\Sigma$ satisfies $N_{3,7}$.  Together with the known behavior for rational normal curves and the conjecture of \cite{op} mentioned above, this suggests the following:

\begin{conj}
For $d\geq2$, $k\geq3$, the secant variety to $v_k(\P^d)$ satisfies $N_{3,3k-5}$.
\qed
\end{conj}

\section{Acknowledgments}
This project grew out of work done together with Jessica Sidman, and benefited greatly from her insight and input, as well as from her comments regarding a preliminary draft of this work.  I would also like to thank Lisa DeMeyer, Hal Schenck, Greg Smith, and Jonathan Wahl for helpful discussion and comments.


\begin{thebibliography}{Namehere}


\bibitem{ah} J. Alexander and A. Hirschowitz, Polynomial interpolation in several variables, J. Algebraic Geom. 4 (1995), no. 2, 201-222.

\bibitem{ar} E. S. Allman, J. A. Rhodes, Phylogenetic ideals and varieties for the general Markov model,  Adv. in Appl. Math.  40  (2008),  no. 2, 127--148. 



\bibitem{bs} Th. Bauer and T. Szemberg, Higher order embeddings of abelian varieties, Math. Z. 224 (1997), no. 3, 449-455. 

\bibitem{bcg} A. Bernardi, M.V. Catalisano, A. Gimigliano, M. Id, Osculating varieties of Veronese varieties and their
higher secant varieties. Canad. J. Math. 59 (2007), no. 3, 488–502.



\bibitem{bertram} A. Bertram, Moduli of Rank-2 Vector Bundles, Theta Divisors, and the Geometry of Curves in Projective Space, J. Diff. Geom. 35 (1992), pp. 429-469.



\bibitem{bel} A. Bertram, L. Ein, and R. Lazarsfeld, Vanishing Theorems, A Theorem of Severi, and the Equations Defining Projective Varieties, J. Amer. Math. Soc. vol. 4 no. 3 (1991), pp. 587-602.

\bibitem{birkenhake} C. Birkenhake, Linear systems on projective spaces, Manuscripta Math. 88 (1995), no. 2, 177–184.


\bibitem{bcr} W. Bruns, A. Conca, and T. Roemer, Koszul homology and syzygies of Veronese subalgebras, arXiv:0902.2431.

\bibitem{BGL} J. Buczynski, A. Ginensky, and J.~M. Landsberg, Determinental equations for secant varieties and the Eisenbud-Koh-Stillman conjecture, arXiv:1007.0192.


\bibitem{CGG05a}
M.~V. Catalisano, A.~V. Geramita, and A.~Gimigliano,
\newblock Higher secant varieties of {S}egre-{V}eronese varieties.
\newblock In {\em Projective varieties with unexpected properties}, pages
  81--107. Walter de Gruyter GmbH \& Co. KG, Berlin, 2005.
  
  \bibitem{CGG05b}
M.~V. Catalisano, A.~V. Geramita, and A.~Gimigliano,
\newblock Higher secant varieties of the {S}egre varieties {$\P\sp
  1\times\dots\times\P\sp 1$}.
\newblock J. Pure Appl. Algebra, 201(1-3):367-380, 2005.

\bibitem{CGG07}
M.~V. Catalisano, A.~V. Geramita, and A.~Gimigliano,
\newblock Segre-{V}eronese embeddings of {$\P\sp 1\times\P\sp
  1\times\P\sp 1$} and their secant varieties.
\newblock Collect. Math., 58(1):1-24, 2007.


\bibitem{CGG} M.~V. Catalisano, A.~V. Geramita, and A.~Gimigliano, On the ideals of secant varieties to certain rational varieties, J. Algebra 319 (2008), no. 5, 1913-1931. 


\bibitem{unex} Projective varieties with unexpected properties, Proceedings of the International Conference ``Projective
Varieties with Unexpected Properties'' held in Siena, June 8–13, 2004, editors Ciliberto, C.,  Geramita, A.
V., Harbourne, B., Mir\'o-Roig, R. M., and Ranestad, K. Walter de Gruyter, Berlin,
2005.

\bibitem{ckk} Y. Choi, P.-L. Kang, and S. Kwak, Higher linear syzygies of inner projections, J. Algebra 305 (2006), no. 2, 859-876. 


\bibitem{conrad} B. Conrad, \textit{Grothendieck Duality and Base Change}, Springer-Verlag Lecture Notes in Mathematics 1750 (2000). 

\bibitem{cs} D. Cox and J. Sidman, Secant varieties of toric varieties, J. Pure Appl. Algebra 209 (2007), no. 3, 651-669. 


\bibitem{el} L. Ein and R. Lazarsfeld, Syzygies and Koszul cohomology of smooth projective varieties of arbitrary dimension, Invent. Math. 111 (1993), no. 1, 51-67. 

\bibitem{eisenbud} D. Eisenbud, Commutative Algebra, With a View Toward Algebraic Geometry, GTM 150, Springer-Verlag, New York, 1995.


\bibitem{EGHP} D. Eisenbud, M. Green, K. Hulek, and S. Popescu, Restricting Linear Syzygies: Algebra and Geometry,  Compos. Math. 141 (2005), pp. 1460-1478.




\bibitem{gss} L. Garcia, M. Stillman, and B. Sturmfels, Algebraic geometry of Bayesian networks, J. Symbolic Comput. 39 (2005), no. 3-4, 331-355.


\bibitem{mgreen} M. Green, Koszul Cohomology and the Geometry of Projective Varieties, J. Diff. Geom. 19 (1984), pp. 125-171.




\bibitem{hart} Robin Hartshorne, {\it Algebraic Geometry}, GTM 52, Springer-Verlag, New York, 1977.

\bibitem{hss} M. Hering, H. Schenck, and G. Smith, Syzygies, multigraded regularity and toric varieties,  Compos. Math. 142 (2006), no. 6, 1499-1506. 

\bibitem{inamdar} S. Inamdar, On syzygies of projective varieties.  Pacific J. Math.  177  (1997),  no. 1, 71-76.


\bibitem{kanev} V. Kanev, Chordal varieties of Veronese varieties and catalecticant matrices, J.Math. Sci. 94 (1999) 1114–1125.

\bibitem{kaw} Y. Kawamata, A generalization of Kodaira-Ramanujam's vanishing theorem, Math. Ann. 261 (1982), no. 1, 43-46. 


\bibitem{kumar} S. Kumar, Proof of Wahl's conjecture on surjectivity of the Gaussian map for flag varieties,
Amer. J. Math. 114 (1992), no. 6, 1201-1220. 


\bibitem{Lan06} J.~M. Landsberg, The border rank of the multiplication of two by two matrices is seven, J. Amer.
Math. Soc. 19 (2006), 447–459.

\bibitem{Lan08}
J.~M. Landsberg,
\newblock Geometry and the complexity of matrix multiplication.
\newblock Bull. Amer. Math. Soc. (N.S.), 45(2):247--284, 2008.

\bibitem{LM08}
J.~M. Landsberg and L.~Manivel,
\newblock Generalizations of {S}trassen's equations for secant varieties of
  {S}egre varieties.
\newblock Comm. Algebra, 36(2):405--422, 2008.

\bibitem{LW07}
J.~M. Landsberg and J. Weyman,
\newblock On the ideals and singularities of secant varieties of {S}egre
  varieties.
\newblock Bull. Lond. Math. Soc., 39(4):685--697, 2007.

\bibitem{LW08} J.~M. Landsberg and J. Weyman, On secant varieties of Compact Hermitian Symmetric Spaces,  J. Pure Appl. Algebra  213  (2009),  no. 11, 2075--2086. 

\bibitem{LO} J.~M. Landsberg and G. Ottaviani, Equations for Secant Varieties to Veronese Varieties, arXiv:1006.0180.


\bibitem{LO2} J.~M. Landsberg and G. Ottaviani, Equations for secant varieties via vector bundles, arXiv:1010.1825.


\bibitem{zak} R. Lazarsfeld and A. Van de Ven, \textit{Topics in the geometry of projective space.}
Recent work of F. L. Zak. With an addendum by Zak. DMV Seminar, 4. Birkhäuser Verlag, Basel, 1984.

\bibitem{M2} Daniel R. Grayson and Michael E. Stillman, Macaulay 2, a software system for research in algebraic geometry, Available at http://www.math.uiuc.edu/Macaulay2.


\bibitem{man} L. Manivel, On the syzygies of flag manifolds, Proc. Amer. Math. Soc. 124 (1996), no. 8, 2293-2299. 



\bibitem{OB} L. Oeding and D.~J. Bates, Toward a salmon conjecture, arXiv:1009.6181.

\bibitem{op} G. Ottaviani and R. Paoletti, Syzygies of Veronese embeddings, Compos. Math. 125 (2001), no. 1, 31-37. 


\bibitem{parwahl} G. Pareschi, Gaussian maps and multiplication maps on certain projective varieties, Compositio Math. 98 (1995), no. 3, 219-268. 

\bibitem{par} G. Pareschi, Syzygies of abelian varieties, J. Amer. Math. Soc. 13 (2000), no. 3, 651-664. 

\bibitem{pp} G. Pareschi and M. Popa, Regularity on abelian varieties I, J. Amer. Math. Soc. 16 (2003), no. 2, 285-302.




\bibitem{rub} E. Rubei, A result on resolutions of Veronese embeddings, Ann. Univ. Ferrara Sez. VII (N.S.) 50 (2004), 151-165. 


\bibitem{sidsul} J. Sidman and S. Sullivant, Prolongations and computational algebra,  Canad. J. Math.  61  (2009),  no. 4, 930--949.

\bibitem{sidver} J. Sidman and P. Vermeire, Syzygies of the Secant Variety,  Algebra Number Theory  3  (2009),  no. 4, 445--465.

\bibitem{ss} B. Sturmfels and S. Sullivant, Combinatorial secant varieties, Pure Appl. Math. Q. 2 (2006), no. 3, 867-891.

\bibitem{thaddeus} M. Thaddeus, Stable pairs, linear systems and the Verlinde formula, Invent. Math. 117 (1994), no. 2, 317-353. 


\bibitem{vermeireflip1} P. Vermeire, Some Results on Secant Varieties Leading to a Geometric Flip Construction, Compositio Mathematica 125 (2001), no. 3, pp. 263-282.

\bibitem{vermeireidealreg} P. Vermeire, On the Regularity of Powers of Ideal Sheaves, Compositio Mathematica, 131 (2002), no. 2, pp. 161-172.

\bibitem{vermeiresecreg} P. Vermeire, Regularity and Normality of the Secant Variety to a Projective Curve, Journal of Algebra 319 (2008), pp. 1264-1270.

\bibitem{vermeiresing} P. Vermeire, Singularities of Secant Varieties, Journal of Pure and Applied Algebra 213 (2009), pp. 1129-1132.

\bibitem{vermeirecurves} P. Vermeire, Equations and syzygies of the first secant variety to a smooth curve, arXiv:0912.0151v1.


\bibitem{vie} E. Viehweg, Vanishing theorems, J. Reine Angew. Math. 335 (1982), 1-8. 



\bibitem{wahl2} J. Wahl, Gaussian maps and tensor products of irreducible representations, Manuscripta Math. 73 (1991), no. 3, 229-259. 

\bibitem{wahl} J. Wahl, On cohomology of the square of an ideal sheaf, J. Algebraic Geom. 6 (1997), no. 3, pp. 481-511.


\end{thebibliography}
\end{document}